\documentclass[11pt, twoside, letter]{article}
\usepackage[margin = 1in]{geometry}

\usepackage[utf8]{inputenc}
\usepackage{amsmath, amssymb, amsfonts}
\usepackage{ifthen}
\usepackage[boxed, linesnumbered]{algorithm2e}
\usepackage{comment}

\usepackage{pgf}
\usepackage{tikz}
\usetikzlibrary{arrows,automata}

\usepackage{microtype}

\usepackage{amsthm}
\usepackage[unicode,psdextra,colorlinks=true,citecolor=blue,bookmarksnumbered=true,hyperfootnotes=false,linkcolor=green!40!black]{hyperref}
\usepackage{multirow, multicol}

\usepackage{subfigure, graphicx}
\newtheorem{theorem}{Theorem}[section]

\newtheorem{claim}[theorem]{Claim}

\newtheorem{lemma}[theorem]{Lemma}
\newtheorem{proposition}[theorem]{Proposition}

\newtheorem{remark}[theorem]{Remark}

\newtheorem{definition}[theorem]{Definition}

\newcommand{\norm}[2][]{\ensuremath{\left\Vert #2 \right\Vert_{#1}}}

\renewcommand{\vec}[1]{\mathbf{#1}}

\begin{document}
% Title portion. Note the short title for running heads
\title{Last-Iterate Convergence: Zero-Sum Games and Constrained Min-Max Optimization}

\author{Constantinos Daskalakis\\MIT\\costis@csail.mit.edu
\and Ioannis Panageas\\MIT\\ioannis@csail.mit.edu}

%% *****************************************
%% *****************************************
% NOTE! Affiliations placed here should be for the institution where the
%       BULK of the research was done. If the author has gone to a new
%       institution, before publication, the (above) affiliation should NOT be changed.
%       The authors 'current' address may be given in the "Author's addresses:" block (below).
%       So for example, Mr. Abdelzaher, the bulk of the research was done at UIUC, and he is
%       currently affiliated with NASA.

\date{}
\maketitle
\begin{abstract}
Motivated by applications in Game Theory, Optimization, and Generative Adversarial Networks, recent work of Daskalakis et al~\cite{DISZ17} and follow-up work of Liang and Stokes~\cite{LiangS18} have established that a variant of the widely used Gradient Descent/Ascent procedure, called ``Optimistic Gradient Descent/Ascent (OGDA)'', exhibits last-iterate convergence to saddle points in {\em unconstrained} convex-concave min-max optimization problems. We show that the same holds true in the more general problem of {\em constrained} min-max optimization under a variant of the no-regret Multiplicative-Weights-Update method called ``Optimistic Multiplicative-Weights Update (OMWU)''. This answers an open question of Syrgkanis et al~\cite{SALS15}.

The proof of our result requires fundamentally different techniques from those that exist in no-regret learning literature and the aforementioned papers. We show that OMWU monotonically improves the Kullback-Leibler divergence of the current iterate to the (appropriately normalized) min-max solution until it enters a neighborhood of the solution. Inside that neighborhood we show that OMWU is locally (asymptotically) stable converging to the exact solution. We believe that our techniques will be useful in the analysis of the last iterate of other learning algorithms. %We experiment with zero-sum games to measure how the convergence rate scales with the dimension.
\end{abstract}

\section{Introduction}
\label{sec:intro}

A central problem in Game Theory and Optimization is computing a pair of probability vectors $(\vec{x},\vec{y})$, solving
\begin{align}
\min_{\vec{y}\in \Delta_m}\max_{\vec{x} \in \Delta_n} \vec{x}^{\top} A \vec{y}, \label{eq:min-max problem}
\end{align}
where $\Delta_n \subset \mathbb{R}^n$ and $\Delta_m \subset \mathbb{R}^m$ are probability simplices, and $A$ is a $n\times m$ matrix. Von Neumann's celebrated minimax theorem informs us that
\begin{align}
\min_{\vec{y}\in \Delta_m}\max_{\vec{x} \in \Delta_n} \vec{x}^{\top} A \vec{y} = \max_{\vec{x} \in \Delta_n} \min_{\vec{y}\in \Delta_m} \vec{x}^{\top} A \vec{y}, \label{eq:minimax theorem}
\end{align}
and that all solutions to the LHS are solutions to the RHS, and vice versa. This result was a founding stone in the development of Game Theory. Indeed, interpreting $\vec{x}^{\top} A \vec{y}$ as the payment of the ``min player'' to the ``max player'' when the former selects a distribution $\vec{y}$ over columns and the latter selects a distribution $\vec x$ over rows of matrix $A$, a solution to~\eqref{eq:min-max problem} constitutes an equilibrium of the game defined by matrix $A$, called a ``minimax equilibrium'', a pair of randomized strategies such that neither player can improve their payoff by unilaterally changing their distribution.

Besides their fundamental value for Game Theory, it is known that~\eqref{eq:min-max problem} and~\eqref{eq:minimax theorem} are also intimately related to Linear Programming. It was shown by von Neumann that~\eqref{eq:minimax theorem} follows from strong linear programming duality. Moreover, it was suggested by Dantzig~\cite{dantzig1951proof} and recently proven by Adler~\cite{A13} that any linear program can be solved by solving some min-max problem of the form~\eqref{eq:min-max problem}. In particular, min-max problems of form~\eqref{eq:min-max problem} are exactly as expressive as min-max problems of the following form, which capture any linear program (by Lagrangifying the constraints):
\begin{align}
\min_{\vec{y} \ge 0}\max_{\vec{x} \ge 0} \left(\vec{x}^{\top} A \vec{y} + b^{\top} \vec{x} + c^{\top} \vec{y}\right). \label{eq:min-max problem positivity constraints}
\end{align}
Soon after the minimax theorem was proven and its connection to linear programming was forged, researchers proposed dynamics for solving min-max optimization problems by having the min and max players of~\eqref{eq:min-max problem} run a simple learning procedure in tandem. An early method, proposed by Brown~\cite{B51} and analyzed by Robinson~\cite{R51}, was fictitious play. Soon after, Blackwell's approachability theorem~\cite{B56} propelled the field of online learning, which lead to the discovery of several learning algorithms converging to minimax equilibrium at faster rates, while also being robust to adversarial environments, situations where one of the players of the game deviates from the prescribed dynamics; see e.g.~\cite{C06}. These learning methods, called ``no-regret'', include the celebrated multiplicative-weights-update method, follow-the-regularized-leader, and follow-the-perturbed-leader. Compared to centralized linear programming procedures the advantage of these methods is the simplicity of executing their steps, and their robustness to adversarial environments, as we just discussed.

\paragraph{Last vs Average Iterate Convergence.} Despite the extensive literature on no-regret learning, an unsatisfactory feature of known results is that min-max equilibrium is shown to be attained only in an average sense. To be precise, if $(\vec{x}^t,\vec{y}^t)$ is the trajectory of a no-regret learning method, it is usually shown that the average  ${1 \over t}\sum_{\tau \le t}{\vec{x}^{\tau}}^{\top} A \vec{y}^{\tau}$ converges to the optimal value of~\eqref{eq:min-max problem}, as $t \rightarrow \infty$. Moreover, if the solution to~\eqref{eq:min-max problem} is unique, then ${1 \over t}\sum_{\tau \le t}(\vec{x}^{\tau}, \vec{y}^{\tau})$ converges to the optimal solution. Unfortunately that does not mean that the last iterate $(\vec{x}^t,\vec{y}^t)$ converges to an optimal solution, and indeed it commonly diverges or enters a limit cycle. Furthermore, in the optimization literature, Nesterov \cite{smooth05} provides a method that can give pointwise convergence (i.e., convergence of the last iterate) to problem~\eqref{eq:min-max problem}\footnote{Nesterov showed that by optimizing $f_\mu (\vec{x}) := \mu\ln (\frac{1}{m} \sum_{j=1}^m e^{-\frac{1}{\mu}(A\vec{x})_j}),$ $g_\nu (\vec{x}) := \nu\ln (\frac{1}{n} \sum_{j=1}^n e^{\frac{1}{\nu}(A^{\top}\vec{y})_j})$ for $\mu = \Theta(\frac{\epsilon}{\log m}),$ $\nu = \Theta(\frac{\epsilon}{\log n})$ yields an $O(\epsilon)$ approximation to the problem problem~\eqref{eq:min-max problem}.}, however his algorithm is not a no-regret learning algorithm.
Recent work by Daskalakis et al~\cite{DISZ17} and Liang and Stokes~\cite{LiangS18} studies whether last iterate convergence can be established for no-regret learning methods in the simple unconstrained min-max problem of the form:
\begin{align}
\min_{\vec{y} \in \mathbb{R}^m}\max_{\vec{x} \in \mathbb{R}^n} \left(\vec{x}^{\top} A \vec{y} + b^{\top} \vec{x} + c^{\top} \vec{y}\right). \label{eq:min-max problem unconstrained}
\end{align}
For this problem, it is known that Gradient Descent/Ascent (GDA) is a no-regret learning procedure, corresponding to follow-the-regularized leader (FTRL) with $\ell_2^2$-regularization. As such, the average trajectory traveled by GDA converges to a min-max solution, in the afore-described sense. On the other hand, it is also known that GDA may diverge from the min-max solution, even in trivial cases such as $A={\rm I}, n=m=1, b=c=0$. Interestingly, \cite{DISZ17,LiangS18} show that a variant of GDA, called ``Optimistic Gradient Descent/Ascent (OGDA),\footnote{OGDA is tantamount to Optimistic FTRL with $\ell_2^2$-regularization, in the same way that GDA is tantamount to FTRL with $\ell_2^2$-regularization; see~e.g.~\cite{RS13}. OGDA essentially boils down to GDA with negative momentum.} exhibits last iterate convergence. Inspired by their theoretical result for the performance of OGDA in~\eqref{eq:min-max problem unconstrained}, Daskalakis et al.~\cite{DISZ17} even propose the use of OGDA for training Generative Adversarial Networks (GANs)~\cite{goodfellow2014generative}. Moreover, Syrgkanis et al.~\cite{SALS15} provide numerical experiments which indicate that the trajectories of Optimistic Hedge (variant of Hedge in the same way OGDA is a variant of GDA) stabilize (i.e., converge pointwise) as opposed to (classic) Hedge and they \textit{posed the question whether Optimistic Hedge actually converges pointwise.}
%GANs are a popular recent approach for learning implicit generative models of high-dimensional distributions by setting up a min-max problem. Typically this problem is solved via GDA, which is well-documented to exhibit training oscillations, a signature of the lack of last iterate convergence; see e.g.~\cite{metz2016unrolled}. While the objective function in GAN training is commonly not convex-concave in the variables of the min and the max players respectively, and is hence out of the scope of the minimax theorem and of strong convex-programming duality, experiments using OGDA for GAN training  by~\cite{DISZ17} suggest that it is better behaved than GDA.

Motivated by the afore-described lines of work, and the importance of last iterate convergence for Game Theory and the modern applications of GDA-style methods in Optimization, our goal in this work is to generalize the results of~\cite{DISZ17,LiangS18} to the general min-max problem~\eqref{eq:min-max problem positivity constraints}, or equivalently~\eqref{eq:min-max problem}; indeed, we will focus on the latter, but our algorithms are readily applicable to the former as the two problems are equivalent~\cite{A13}. With the constraint that $(\vec{x}, \vec{y})$ should remain in  $\Delta_n \times \Delta_m$, GDA and OGDA are not applicable. Indeed, the natural GDA-style method for min-max problems in this case is the celebrated Multiplicative-Weights-Update (MWU) method, which is tantamount to FTRL with entropy-regularization. Unsurprisingly, in the same way that GDA suffers in the unconstrained problem~\eqref{eq:min-max problem unconstrained}, MWU exhibits cycling in the constrained problem~\eqref{eq:min-max problem} (a recent work is \cite{BP18} and was also shown empirically in~\cite{SALS15}). So it is natural for us to study instead its optimistic variant, ``Optimistic Multiplicative-Weights-Update (OMWU),'' (called Optimistic Hedge in~\cite{SALS15}) which corresponds to Optimistic FTRL with entropy-regularization, the equations of which are given in Section~\ref{sec:definitions}. Our main result is the following (restated as Theorem~\ref{thm:mainomwu} after Section~\ref{sec:definitions}) and answers an open question asked in~\cite{SALS15} as applicable to two player zero sum games:
\begin{theorem}[\textbf{Last-Iterate Convergence of OMWU}]\label{thm:informal theorem}  Whenever~\eqref{eq:min-max problem} has a unique optimal solution $(\vec{x}^*,\vec{y}^*)$, OMWU with appropriate choice of learning rate and initialized at the pair of uniform distributions $({1\over n}{\bf 1},{1 \over m}{\bf 1})$ exhibits last-iterate convergence to the optimal solution. That is, if $(\vec{x}^t,\vec{y}^t)$ are the vectors maintained by OMWU at step $t$, then $\lim_{t \to \infty}(\vec{x}^t,\vec{y}^t) = (\vec{x}^*,\vec{y}^*)$.
\end{theorem}
\begin{remark} We note that the assumption about uniqueness of the optimal solution for problem~\eqref{eq:min-max problem} is generic in the following sense: Within the set of all zero-sum games, the set of zero-sum games with non-unique equilibrium has Lebesgue measure zero~\cite{BP18,EVD}. This implies that if $A$'s entries are sampled independently from some continuous distribution, then with probability one the min-max problem~\eqref{eq:min-max problem} will have a unique solution.
\end{remark}
Our paper provides \textit{two important messages}:
\begin{itemize}
\item It strengthens the intuition that optimism helps the trajectories of learning dynamics stabilize (e.g., Optimistic MWU vs MWU or Optimistic GDA vs GDA; as the papers of Syrgkanis et al~\cite{SALS15} and Daskalakis et al~\cite{DISZ17} also do).
    \item The techniques we use (typically appear in dynamical systems literature) to prove convergence for the last iterate, are fundamentally different from those commonly used to prove convergence of the time average of a learning algorithm.
\end{itemize}
\textbf{Notation:} Vectors in $\Delta_n, \Delta_m$ are denoted in boldface $\vec{x}, \vec{y}$. Time indices are denoted by superscripts. Thus, a time indexed vector $\vec{x}$ at time $t$ is denoted as $\vec{x}^t$. We use the letter $J$ to denote the Jacobian of a function (with appropriate subscript), $\vec{I}, \vec{0}, \vec{1}$ to denote the identity, zero matrix and all ones vector respectively with appropriate subscripts to indicate the size. Moreover, $(A\vec{y})_i$ captures $\sum_j A_{ij}y_j$. The support of $\vec{x}$ is denoted by $\textrm{Supp}(\vec{x})$. Finally we use $(\vec{x}^*,\vec{y}^*)$ to denote the optimal solution for the min-max problem (\ref{eq:min-max problem}) and $[n]$ to denote $\{1,...,n\}$.

\section{Preliminaries}
\label{sec:prelims}

\subsection{Definitions and facts}\label{sec:moredef}

\paragraph{Dynamical Systems.}
A recurrence relation of the form $\vec{x}^{t+1} = w(\vec{x}^{t})$ is a discrete time dynamical system, with update rule $w:\mathcal{S} \to \mathcal{S}$ where $\mathcal{S} = \Delta_n \times \Delta_m \times \Delta_n \times \Delta_m$ for our purposes.
The point $\vec{z}$ is called a \textit{fixed point} or \textit{equilibrium} of $w$ if $w(\vec{z}) = \vec{z}$. We will be interested in the following well known fact that will be used in our proofs.
\begin{proposition}[e.g.~\cite{G07}]\label{prop:contraction}
If the Jacobian of the update rule $w$\footnote{We assume $w$ is a continuously differential function.} at a fixed point $\vec{z}$ has spectral radius less than one, then there exists a neighborhood $U$ around $\vec{z}$ such that for all $\vec{x} \in U$, the dynamics converges to $\vec{z}$, i.e., $\lim_{n \to \infty} w^n(\vec{x}) = \vec{z}$. We call $w$ an asymptotic stable mapping in $U$.
\end{proposition}

\subsection{OMWU Method}\label{sec:definitions}
Our main contribution is that the last iterate of OMWU converges to the optimal solution. The OMWU dynamics is defined as follows ($t \geq 1$):
\begin{equation}\label{eq:expOMWU}
\begin{array}{ll}
x_i^{t+1} = x_i^t \frac{e^{ 2\eta (A\vec{y}^t)_i -\eta  (A\vec{y}^{t-1})_i}}{\sum_{j=1}^n x_j^t e^{ 2\eta (A\vec{y}^t)_j -\eta  (A\vec{y}^{t-1})_j}} \textrm{ for all }i\in[n],\\
y_i^{t+1} = y_i^t \frac{e^{ -2\eta (A^{\top}\vec{x}^t)_i +\eta  (A^{\top}\vec{x}^{t-1})_i}}{\sum_{j=1}^m y_j^t e^{ -2\eta (A^{\top}\vec{x}^t)_j +\eta  (A^{\top}\vec{x}^{t-1})_j}} \textrm{ for all }i\in[m].
\end{array}
\end{equation}
Points $(\vec{x}^1,\vec{y}^1), (\vec{x}^0,\vec{y}^0)$ are the initial conditions and are given as input. We call $0<\eta<1$ the {\em stepsize} of the dynamics. It is more convenient to interpret OMWU dynamics as mapping a quadruple to quadruple
($(\vec{x}^{t},\vec{y}^t,\vec{x}^{t-1},\vec{y}^{t-1}) \to (\vec{x}^{t+1},\vec{y}^{t+1},\vec{x}^{t},\vec{y}^{t})$, see Section~\ref{sec:contraction} for the construction of the dynamical system).
\begin{remark}
Let $(\vec{x}^*,\vec{y}^*)$ be the optimal solution. We see that $(\vec{x}^*,\vec{y}^*,\vec{x}^*,\vec{y}^*)$ is a fixed point of the mapping. Furthermore, $\Delta_n \times \Delta_m \times \Delta_n \times \Delta_m$ is invariant under OMWU dynamics. For $t\geq 1$, if $x_i^{t}=0$ then $x_{i}$ remains zero for all times greater than $t$, and if it is positive, it remains positive (both numerator and denominator are positive) \footnote{Same holds for vector $\vec{y}$.}. In words, at all times the OMWU satisfies the non-negativity constraints and the renormalization factor (denominator) makes both $\vec{x},\vec{y}$'s coordinates sum up to one.
A last observation is that every fixed point of OMWU dynamics (mapping a quadruple to quadruple) has the form $(\vec{x},\vec{y},\vec{x},\vec{y})$ (two same copies). Equation (\ref{eq:mysystem}) shows how to express OMWU dynamics as a dynamical system.
\end{remark}

\subsection{Linear Variant of OMWU}
We provide the linear variant of OMWU dynamics (\ref{eq:expOMWU}) because we use it in some intermediate lemmas (appear in appendix).
\begin{equation}\label{eq:linOMWU}
\begin{array}{ll}
x_i^{t+1} = x_i^t \frac{1+ 2\eta (A\vec{y}^t)_i -\eta  (A\vec{y}^{t-1})_i}{\sum_{j=1}^n x_j^t (1+ 2\eta (A\vec{y}^t)_j -\eta  (A\vec{y}^{t-1})_j)} \textrm{ for all }i\in[n] ,\\
y_i^{t+1} = y_i^t \frac{1 -2\eta (A^{\top}\vec{x}^t)_i +\eta  (A^{\top}\vec{x}^{t-1})_i}{\sum_{j=1}^m y_j^t (1-2\eta (A^{\top}\vec{x}^t)_j +\eta  (A^{\top}\vec{x}^{t-1})_j)} \textrm{ for all }i\in[m].
\end{array}
\end{equation}
This dynamics is derived by considering the first order approximation of the exponential function. Stepsize $\eta$ in this case should be chosen sufficiently small so that both numerator and denominator are positive.

\subsection{More definitions and statement of our result}
\begin{definition}[\cite{MPPTV17}]\label{def:close} Assume $\alpha>0$. We call a point $(\vec{x},\vec{y}) \in \Delta_n \times \Delta_m$ $\alpha$-close if for each $i$ we have that $x_i \leq \alpha$ or $|\vec{x}^{\top}A\vec{y} - (A\vec{y})_i| \leq \alpha$ and for each $j$ it holds $y_j \leq \alpha$ or $|\vec{x}^{\top}A\vec{y} - (A^{\top}\vec{x})_j| \leq \alpha$.
\end{definition}
\begin{remark} Think of $\alpha$-close points as $\alpha$-approximate optimal solutions for min-max problems that are induced by {\em submatrices} of $A$ ($\alpha$-approximate stationary points). Moreover, if $(\vec{x},\vec{y})$ is $0$-close point does not necessarily imply $(\vec{x},\vec{y})$ is the optimal solution of problem (\ref{eq:min-max problem})!
\end{remark}

\begin{definition}[Approximate solution]\label{def:approximate} Assume $\epsilon>0$. We call a point $(\vec{x},\vec{y}) \in \Delta_n \times \Delta_m$ $\epsilon$-approximate (or $\epsilon$-approximate Nash equilibrium) if for all $\tilde{\vec{x}} \in \Delta_n$ we get that $\tilde{\vec{x}}^{\top}A\vec{y} \leq \vec{x}^{\top}A\vec{y} + \epsilon$ (max player deviates) and for all $\tilde{\vec{y}} \in \Delta_m$ we get that $\vec{x}^{\top}A\tilde{\vec{y}} \geq \vec{x}^{\top}A\vec{y} - \epsilon$ (min player deviates).
\end{definition}
\begin{remark} Think of $\epsilon$-approximate points as approximate optimal solutions to the min-max problem (\ref{eq:min-max problem}). Moreover, if $(\vec{x},\vec{y})$ is $0$-approximate then $(\vec{x},\vec{y})$ is the optimal solution of problem (\ref{eq:min-max problem}).
\end{remark}

\paragraph{Statement of our results.} We finish the preliminary section by stating formally the main result.

\begin{theorem}[OMWU converges]\label{thm:mainomwu} Let $A$ be a $n \times m$ matrix and assume that
\[
\min_{\vec{y} \in \Delta_m} \max_{\vec{x} \in \Delta_n} \vec{x}^{\top} A \vec{y}
\]
has a unique solution $(\vec{x}^*,\vec{y}^*)$. It holds that for $\eta$ sufficiently small (depends on $n,m,A$), starting from the uniform distribution, i.e.,  $(\vec{x}^1,\vec{y}^1) = (\vec{x}^0,\vec{y}^0)=({1\over n}{\bf 1},{1 \over m}{\bf 1})$, it holds \[\lim_{t \to \infty}(\vec{x}^t,\vec{y}^t) = (\vec{x}^*,\vec{y}^*),\] under OMWU dynamics. The stepsize $\eta$ is constant, i.e., does not vanish with time\footnote{Our proof also works if the starting points $(\vec{x}^1,\vec{y}^1), (\vec{x}^0,\vec{y}^0)$ are both in the interior of $\Delta_n \times \Delta_m$ and not necessarily uniform, however the choice of $\eta$ depends on the initial distributions as well and not only on $n,m,A$.}.
\end{theorem}
We need to note that it is not clear from our theorem how small $\eta$ is and its dependence on the size of $A$. Moreover, OMWU has two phases (the phase where KL divergence decreases and the local asymptotic stability phase, see theorems below) where the stepsize is constant but it might change when we move from phase one to phase two. Nevertheless, our convergence result holds for constant stepsizes as opposed to the classic no-regret learning literature where $\eta$ scales like $\frac{1}{\sqrt{T}}$ after $T$ iterations. Another result we know of this flavor is about MWU algorithm on congestion games \cite{PPP17}.

\section{Last iterate convergence of OMWU}
In this section we show our main result (Theorem \ref{thm:mainomwu}), by breaking the proof into three key theorems. The first theorem says that KL divergence from the $t$-th iterate $(\vec{x}^t,\vec{y}^t)$ to the optimal solution $(\vec{x}^*,\vec{y}^*)$, i.e., (sum of KL divergences to be exact) \[\sum_i x_i^* \ln (x_i^*/x_i^t)+\sum_i y_i^* \ln (y_i^*/y_i^t),\] decreases with time $t \geq 2$ by at least a factor of $\eta^3$ per iteration, unless the iterate $(\vec{x}^t,\vec{y}^t)$ is $O(\eta^{1/3})$-close (see Definition \ref{def:close}). Moreover, provided that the stepsize $\eta$ is small enough, we can show the structural result that $(\vec{x}^t,\vec{y}^t)$ lies in a neighborhood of $(\vec{x}^*,\vec{y}^*)$ that becomes smaller and smaller as $\eta \to 0$. Finally, as long as OMWU dynamics has reached a small neighborhood around $(\vec{x}^*,\vec{y}^*)$, we show that the update rule of the dynamical system induced by OMWU is locally (asymptotically) stable (for maybe different choice of learning rate), and the last iterate convergence result follows. Formally we show:
\begin{theorem}[KL decreasing]\label{thm:kl} Let $(\vec{x}^*,\vec{y}^*)$ be the unique optimal solution of problem (\ref{eq:min-max problem}) and $\eta$ sufficiently small. Then
\[D_{KL}((\vec{x}^*,\vec{y}^*) || (\vec{x}^t,\vec{y}^t))\] is decreasing with time $t$ by (at least) $\Omega(\eta^3)$ unless $(\vec{x}^t,\vec{y}^t)$ is $O(\eta^{1/3})$-close.
\end{theorem}
\begin{theorem}[$\eta^{1/3}$-close implies close to optimum in $\ell_1$]\label{thm:smalleta} Assume that $(\vec{x}^*,\vec{y}^*)$ is unique optimal solution of the problem (\ref{eq:min-max problem}). Let $T$ (depends on $\eta$) be the first time KL divergence does not decrease by $\Omega(\eta^3)$.
It follows that as $\eta \to 0$, the $\eta^{1/3}$-close point $(\vec{x}^ {T},\vec{y}^{T})$ has distance from $(\vec{x}^*,\vec{y}^*)$ that goes to zero, i.e., $\lim_{\eta \to 0}\norm[1]{(\vec{x}^*,\vec{y}^*) - (\vec{x}^T,\vec{y}^T)} = 0$.
\end{theorem}
\begin{theorem}[OMWU is a locally converging] \label{thm:omwucontraction} Let $(\vec{x}^*,\vec{y}^*)$ be the unique optimal solution to the min-max problem (\ref{eq:min-max problem}). There exists a neighborhood $U:= U(\eta) \subset \Delta_n \times \Delta_m \times \Delta_n \times \Delta_m$ of $(\vec{x}^*,\vec{y}^*,\vec{x}^*,\vec{y}^*)$\footnote{Since $(\vec{x}^*,\vec{y}^*,\vec{x}^*,\vec{y}^*)$ might be on the boundary of $\Delta_n\times \Delta_m \times \Delta_n\times \Delta_m$, $U$ is the intersection of an open ball around $(\vec{x}^*,\vec{y}^*,\vec{x}^*,\vec{y}^*)$ with $\Delta_n\times \Delta_m \times \Delta_n\times \Delta_m$.} so that for all $(\vec{x}^1,\vec{y}^1,\vec{x}^0,\vec{y}^0) \in U$ we have that $\lim_{t \to \infty} (\vec{x}^t,\vec{y}^t,\vec{x}^{t-1},\vec{y}^{t-1}) = (\vec{x}^*,\vec{y}^*,\vec{x}^*,\vec{y}^*)$ under OMWU dynamics as defined in (\ref{eq:expOMWU}) and (\ref{eq:mysystem}) (Section \ref{sec:contraction}).
\end{theorem}
Assuming these three theorems, our main result is straightforward.
\begin{proof}[Proof of Theorem \ref{thm:mainomwu}]  Let $\eta$ be sufficiently small ($\eta$ is the stepsize of the first phase of OMWU when KL decreases). If $(\vec{x}^1, \vec{y}^1) = ({1\over n}{\bf 1},{1 \over m}{\bf 1})$ (starting point is uniform) then an easy upper bound (by removing negative terms) on KL divergence from $(\vec{x}^1,\vec{y}^1)$ to $(\vec{x}^*,\vec{y}^*)$ is $-\sum_{i=1}^n x^*_{i} \log x^1_i + \sum_{i=1}^m y^*_{i} \log y^1_i = \log (nm)$. Therefore using Theorem \ref{thm:kl} we have that after at most $T$ that is  $O(\frac{\log (nm)}{\eta^3})$ steps, OMWU reaches a $O(\eta^{1/3})$-close point ($T$ is the first time so that KL divergence from current iterate to optimal solution $(\vec{x}^*,\vec{y}^*)$ has not decreased by at least a factor of $\eta^3$) or the KL divergence between the optimal solution and $(\vec{x}^{T},\vec{y}^{T})$ is $O(\eta^3)$ (KL divergence was decreasing by at least a factor of $\eta^3$ for all iterations until the
iterate reached a $\ell_1$ distance $O(\eta^3)$). In the latter case it follows $\norm[1]{(\vec{x}^*,\vec{y}^*) - (\vec{x}^T,\vec{y}^T)}^2$ is $O(\eta^3)$ and hence $(\vec{x}^T,\vec{y}^T)$ is $O(\eta^{3/2})$ in $\ell_1$ distance from the optimal solution, therefore for small $\eta$, $(\vec{x}^{T+1},\vec{y}^{T+1}, \vec{x}^T,\vec{y}^T)$ is in the neighborhood $U(\eta')$ that is needed for asymptotic stability (Theorem \ref{thm:omwucontraction}, for appropriate choice of $\eta'$). In the former case, by Theorem \ref{thm:smalleta} (for $\eta$ sufficiently small) it follows that $(\vec{x}^{T+1},\vec{y}^{T+1},\vec{x}^T,\vec{y}^T)$ is also in the neighborhood $U(\eta')$ that is needed for local asymptotic stability (Theorem \ref{thm:omwucontraction})\footnote{In both cases we used that iterate $(\vec{x}^T,\vec{y}^T)$ and $(\vec{x}^{T+1},\vec{y}^{T+1})$ have $\ell_1$ distance $O(\eta)$, this is Lemma~\ref{lem:sameorder}.}. The proof follows by Theorem \ref{thm:omwucontraction} as long as we change the stepsize from $\eta$ to $\eta'$ (in the second phase).
\end{proof}
In the next subsections we will provide the proofs to all three key theorems.

\subsection{KL decreases and OMWU reaches neighborhood}
In this subsection we argue about the proofs of Theorems \ref{thm:kl} and \ref{thm:smalleta}.
The inequality we managed to prove (see in the appendix the proof of Theorem \ref{thm:kl}) is the following:
\begin{equation}\label{eq:klgradient}
\begin{array}{lll}
D_{KL}((\vec{x}^*,\vec{y}^*) || (\vec{x}^{t+1},\vec{y}^{t+1})) -  D_{KL}((\vec{x}^*,\vec{y}^*) || (\vec{x}^{t},\vec{y}^{t}))\leq\\
 -\sum_{i=1}^n x_i^{t} \left((\frac{1}{2} - O(\eta))\eta^2 \left(2 (A\vec{y}^t)_i - 2 \vec{x}^{t \ \top}A\vec{y}^t -(A\vec{y}^{t-1})_i+ \vec{x}^{t \ \top}A\vec{y}^{t-1}\right)^2\right)
\\ -\sum_{i=1}^m y_i^{t} \left((\frac{1}{2} - O(\eta))\eta^2 \left(2 (A^{\top}\vec{x}^t)_i - 2 \vec{x}^{t \ \top}A\vec{y}^t -(A^{\top}\vec{x}^{t-1})_i+ \vec{x}^{t-1 \ \top}A\vec{y}^{t}\right)^2\right)+O(\eta^3).
 \end{array}
\end{equation}
The proof of the inequality is quite long, we choose to provide intuition and skip the details. We refer to the appendix for a proof. The inequality says that OMWU dynamics has a good progress (KL divergence decreases by at least a factor of $\eta^3$) as long as the current and previous iterate $(\vec{x}^t,\vec{y}^t), (\vec{x}^{t-1},\vec{y}^{t-1})$ are not $\alpha$-close for $\alpha$ chosen to be $O(\eta^{1/3})$. This situation appears a lot in gradient methods when the dynamics is close to a stationary point, the gradient of $f$ is small and the progress is small as opposed to the case where the gradient of $f$ is big and there is satisfying progress. The RHS of inequality (\ref{eq:klgradient}) captures the ``distance" from stationarity. Thus, as long as we are not close to a stationary point (i.e., $O(\eta^{1/3})$-close) in a time window between 1,2,...,k, KL divergence from current iterate ($k$-th) to the optimum has decreased by (at least) $\Omega(k\eta^3)$ compared to KL divergence from first iterate to the optimum.

Moreover, suppose that at some point of OMWU dynamics, KL divergence from current iterate to the optimum did not decrease by at least a factor of $\eta^3$ and let $T$ be the iteration this happened. As we have already argued, $(\vec{x}^T,\vec{y}^T)$ is a $O(\eta^{1/3})$-close point. We can show that as long as $\eta$ is sufficiently small, then for all $i,j$ in the support of $(\vec{x}^{*},\vec{y}^*)$, $x_i^T, y_j^T$ are (at least) $\Omega (\eta^{1/3})$ i.e., coordinates in the support of the optimum will have non negligible probability in $(\vec{x}^T,\vec{y}^T)$. Formally:

\begin{lemma}\label{lemma1} Let $i \in \textrm{Supp}(\vec{x}^*)$ and $j \in \textrm{Supp}(\vec{y}^*)$. It holds that $x_i^T \geq \eta^{1/4}$ and $y_i^T \geq  \eta^{1/4}$ as long as \[\eta^{1/4} \ll \min_{s \in \textrm{Supp}(\vec{x}^*)} \frac{1}{(nm)^{1/x_s^*}}, \min_{s \in \textrm{Supp}(\vec{y}^*)} \frac{1}{(nm)^{1/y_s^*}}.\]
\end{lemma}
\begin{proof} By definition of $T$, the KL divergence is decreasing for $2 \leq t \leq T-1$, thus $$D_{KL}((\vec{x}^*,\vec{y}^*) || (\vec{x}^{T-1},\vec{y}^{T-1})) < D_{KL}((\vec{x}^*,\vec{y}^*) || (\vec{x}^1,\vec{y}^1)).$$
Therefore $x_i^* \log \frac{1}{x_i^{T-1}} < \sum_i x_i^* \log \frac{1}{x^1_i}+\sum_i y_j^* \log \frac{1}{y^1_j} = \log (mn)$.
It follows that $x_i^T > 1/(mn)^{\frac{1}{x_i^*}} \geq \eta^{1/4}$ for $x_i^* >0$ ($i \in \textrm{Supp}(\vec{x}^*)$). Since $|x_i^T-x_i^{T-1}|$ is $O(\eta)$ (Lemma \ref{lem:sameorder}) the result follows.
Similarly, the argument works for $y_j^{T}$.
\end{proof}

Lemma \ref{lemma1} indicates that the stepsize $\eta$ might have to be exponentially small in the dimension (OMWU dynamics is slow when $\eta$ is very small).  We can now prove Theorem \ref{thm:smalleta}.
\begin{proof}[Proof of Theorem \ref{thm:smalleta}]
To prove our claim, we are going first to show that every strategy that is not in the support of the unique minmax solution $(\vec{x}^*,\vec{y}^*)$ should have probability mass $O(\eta^{1/4}).$ 

From Lemma \ref{lemma1}, the definition of $T$ and because $\eta^{1/4}\gg\eta^{1/3}$, we get that $|(A\vec{y}^T)_i - \vec{x}^{T \ \top}A\vec{y}^T|$ is $O(\eta^{1/3})$ for all $i$ in the support of $\vec{x}^*$ and $|(A^{\top}\vec{x}^T)_j - \vec{x}^{T \ \top}A\vec{y}^T|$ is $O(\eta^{1/3})$ for all $j$ in the support of $\vec{y}^*$. We consider $(\vec{w}^T , \vec{z}^T)$ to be the ``projection" of the point $(\vec{x}^T,\vec{y}^T)$ by removing all the coordinates that have probability mass less than $\frac{1}{2}\eta^{1/4}$ and rescale so that the coordinates sum up to one.

We restrict ourselves to the corresponding subproblem (submatrix); let's call the corresponding payoff matrix of the subproblem $\tilde{A}$. It is clear that $(\vec{w}^T , \vec{z}^T)$ is a $O(\eta^{1/4})$-approximate solution \footnote{By $\epsilon$-approximate solution we mean the $\epsilon$-approximate Nash equilibrium notion (additive), see Definition \ref{def:approximate}.} for the subproblem with payoff matrix $\tilde{A}$. Let $v = \tilde{\vec{x}}^{*\;\top} \tilde{A} \tilde{\vec{y}}^* = \vec{x}^{*\;\top}A\vec{y}^*$ be the minmax value and $(\tilde{\vec{x}}^*,\tilde{\vec{y}}^*)$ the minmax solution of the subproblem ($(\tilde{\vec{x}}^*,\tilde{\vec{y}}^*)$ has the same non-zero entries as vector $(\vec{x}^*,\vec{y}^*)$). By uniqueness of the optimal solution, we get that $(\tilde{A}\tilde{\vec{y}}^*)_i = v$ for all $i \in \textrm{Supp}(\tilde{\vec{x}}^*)$ and $(\tilde{A}\tilde{\vec{y}}^*)_i < v$ otherwise (check Lemma C.3 in paper \cite{MPP18} for a proof, where they use Farkas' lemma to show it, we use this fact later in Section \ref{sec:contraction}). Similarly $(\tilde{A}^{\top}\tilde{\vec{x}}^*)_j = v$ for the min player $\tilde{\vec{y}}$ if $j$ lies in the support of $\tilde{\vec{y}}^*$ and $(\tilde{A}^{\top}\tilde{\vec{x}}^*)_j > v$ otherwise. We choose $\eta$ so small that every $O(\eta^{1/4})$-approximate solution $(\vec{p},\vec{q})$ has the property that $(\tilde{A}\vec{q})_i \leq v - \eta^{1/5}$, $(\tilde{A}^{\top}\vec{p})_j \geq v + \eta^{1/5}$ for all $i \notin \textrm{Supp}(\tilde{\vec{x}}^*)$ and $j \notin \textrm{Supp}(\tilde{\vec{y}}^*)$ respectively (this is possible by continuity of the bilinear function and Claim \ref{lem:claim3} below). 

Hence we conclude that for $\eta$ small enough, the coordinates in the vector $(\vec{w}^T,\vec{z}^T)$ that are not in the support of the optimal solution $(\tilde{\vec{x}}^*,\tilde{\vec{y}}^*)$ $-$ since $\eta^{1/5}\gg\eta^{1/4}$ $-$ should have probability mass $O(\eta^{1/4})$ at time $T$.

\begin{claim}\label{lem:claim3} Let $(\vec{x}^*,\vec{y}^*)$ be the unique optimal solution to the problem (\ref{eq:min-max problem}). For every $\epsilon>0$, there exists an $\delta(\epsilon)>0$ so that for every $\delta$-approximate solution $(\vec{x},\vec{y})$ we get that $|x_i - x_i^*|<\epsilon$ for all $i \in [n]$. Analogously holds for player $\vec{y}$.
\end{claim}
\begin{proof} We will prove this by contradiction. Assume there is an $\epsilon$ that violates this statement. We choose a sequence $\delta_k$ so that $\lim_{k\to \infty}\delta_k = 0$ and also there is a sequence $(\vec{x}_k,\vec{y}_k)$ of $\delta_k$-approximate Nash equilibrium with $|x_{k,i}-x^*_i| \geq \epsilon$
for some strategy $i$. Since $\Delta_n\times \Delta_m$ is compact and the sequence above is bounded, there is a convergent subsequence. The limit of the convergent subsequence is a Nash equilibrium by definition of $\delta$-approximate (Definition \ref{def:approximate}). By uniqueness it follows that the $i$-th coordinate of the convergent sequence must converge to $x_i^*$, hence we reached a contradiction.
\end{proof}

Therefore, if we restrict to the subgame with payoff matrix $\tilde{A}$, the projected vector $(\vec{w}^T,\vec{z}^T)$ is a $O(\eta^{1/3})$-approximate minmax solution of the subgame.

From Claim \ref{lem:claim3}, as $\eta \to 0$ it follows that the $\ell_1$ distance (any distance suffices) between $(\vec{w}^T,\vec{z}^T)$ and the Nash equilibrium $(\tilde{\vec{x}}^*,\tilde{\vec{y}}^*)$ of the subgame goes to zero. Since the minmax solution of the subgame is the effectively the same as the optimal solution of the original game we get that as $\eta \to 0$,
 $(\vec{x}^T,\vec{y}^T)$ reaches $(\vec{x}^*,\vec{y}^*)$. In particular, since $\norm[1]{(\vec{x}^{T+1},\vec{y}^{T+1})- (\vec{x}^{T},\vec{y}^T)}$ is $O(\eta)$ (see Lemma \ref{lem:sameorder}) there exists a $\eta$ small so that $(\vec{x}^{T+1},\vec{y}^{T+1},\vec{x}^T,\vec{y}^T)$ is inside the necessary neighborhood
$U$ of $(\vec{x}^*,\vec{y}^*,\vec{x}^*,\vec{y}^*)$ that gives local (asymptotic) stability (Theorem \ref{thm:omwucontraction}).
\end{proof}

\subsection{Proving local convergence}\label{sec:contraction}
The purpose of this section is to prove Theorem \ref{thm:omwucontraction}. First of all, we assume that the stepsize $\eta>0$ is some fixed constant (sufficiently small, not necessarily the same stepsize as in the first phase where KL divergence decreases). To show asymptotic stability of OMWU dynamics in a neighborhood of the optimal solution $(\vec{x}^*,\vec{y}^*)$, we first construct a dynamical system that captures OMWU. Moreover, we prove that the Jacobian of the update rule of that particular dynamical system computed at the optimal solution, has spectral radius less than one. This suffices to prove asymptotic stability (see Proposition \ref{prop:contraction}). As a result, as long as OMWU reaches a small neighborhood of $(\vec{x}^*,\vec{y}^*,\vec{x}^*,\vec{y}^*)$, it converges pointwise (last iterate convergence) to it\footnote{Since the dynamical system is from a quadruple to a quadruple, it is a neighborhood of $(\vec{x}^*,\vec{y}^*,\vec{x}^*,\vec{y}^*)$.}. Below we provide the update rule $g$ of the dynamical system, which consists of 4 components:
\begin{equation}\label{eq:mysystem}
\begin{array}{llll}
g(\vec{x},\vec{y},\vec{z},\vec{w}) := (g_1(\vec{x},\vec{y},\vec{z},\vec{w}),g_2(\vec{x},\vec{y},\vec{z},\vec{w}),g_3(\vec{x},\vec{y},\vec{z},\vec{w}),g_4(\vec{x},\vec{y},\vec{z},\vec{w})),\\
g_{1,i}(\vec{x},\vec{y},\vec{z},\vec{w}) := (g_1(\vec{x},\vec{y},\vec{z},\vec{w}))_ i := x_i \frac{e^{2\eta(A\vec{y})_i - \eta(A\vec{w})_i}}{\sum_{t} x_t e^{2\eta(A\vec{y})_t - \eta(A\vec{w})_t}} \textrm{ for all }i \in [n],\\
g_{2,i}(\vec{x},\vec{y},\vec{z},\vec{w}) := (g_2(\vec{x},\vec{y},\vec{z},\vec{w}))_i := y_i \frac{e^{-2\eta(A^{\top}\vec{x})_i + \eta(A^{\top}\vec{z})_i}}{\sum_{t} y_t e^{-2\eta(A^{\top}\vec{x})_t + \eta(A^{\top}\vec{z})_t}} \textrm{ for all }i \in [m],\\
g_3(\vec{x},\vec{y},\vec{z},\vec{w}):= \vec{I}_{n\times n}\vec{x},\\
g_4(\vec{x},\vec{y},\vec{z},\vec{w}):= \vec{I}_{m\times m}\vec{y}.
\end{array}
\end{equation}
It is not hard to check that \[(\vec{x}_{t+1},\vec{y}_{t+1},\vec{x}_{t},\vec{y}_{t}) = g(\vec{x}_{t},\vec{y}_{t},\vec{x}_{t-1},\vec{y}_{t-1}),\]
so $g$ captures exactly the dynamics of OMWU (\ref{eq:expOMWU}). The equations of the Jacobian of $g$ can be found in the appendix (see Section \ref{sec:equationsJacobian}).

\paragraph{Spectral analysis the Jacobian of OMWU at the optimal solution.}
The rest of the section constitutes the proof of Theorem \ref{thm:omwucontraction}. Assume $v = \vec{x}^{* \ \top}A \vec{y}^*$, i.e., $v$ is the value of the bilinear function $\vec{x}^{\top}A\vec{y}$ at the optimal solution. We will analyze the Jacobian computed at $(\vec{x}^*,\vec{y}^*,\vec{x}^*,\vec{y}^*)$\footnote{See also Equations~(\ref{eq:Jacobian2}) of the Jacobian computed at $(\vec{x}^*,\vec{y}^*,\vec{x}^*,\vec{y}^*)$.}.

Assume $i \notin \textrm{Supp}(\vec{x}^*)$, then \[\frac{\partial g_{1,i}}{\partial x_i} = \frac{e^{\eta(A\vec{y}^*)_i} }{\sum x^*_t e^{\eta(A\vec{y}^*)_t} } = \frac{e^{\eta(A\vec{y}^*)_i} }{e^{\eta v}}\] and all other partial derivatives of $g_{1,i}$ are zero, thus $\frac{e^{\eta(A\vec{y}^*)_i} }{e^{\eta v}}$ is an eigenvalue of the Jacobian computed at $(\vec{x}^*,\vec{y}^*,\vec{x}^*,\vec{y}^*)$. Moreover because of uniqueness of the optimal solution, it holds that $\frac{e^{\eta(A\vec{y}^*)_i} }{e^{\eta v}}<1$ because $(A\vec{y}^*)_i - v <0$ (check Lemma C.3 in \cite{MPP18} for a proof, where they use Farkas' Lemma to show it). Similarly, it holds for $j \notin \textrm{Supp}(\vec{y}^*)$ that $\frac{\partial g_{2,j}}{\partial y_j} = \frac{e^{-\eta(A^{\top}\vec{x}^*)_j }}{e^{-\eta v}}<1$ (again by C.3 in \cite{MPP18} it holds that $(A\vec{x}^*)_j - v >0$) and all other partial derivatives of $g_{2,j}$ are zero, hence $\frac{e^{-\eta(A^{\top}\vec{x}^*)_j }}{e^{-\eta v}}$ is an eigenvalue of the Jacobian computed at the optimal solution.

Let $D_x$ be the diagonal matrix of size $|\textrm{Supp}(\vec{x}^*)| \times |\textrm{Supp}(\vec{x}^*)|$ that has on the diagonal the nonzero entries of $\vec{x}^*$ and similarly we define $D_y$ of size $|\textrm{Supp}(\vec{y}^*)| \times |\textrm{Supp}(\vec{y}^*)|$. We set $k_1 = |\textrm{Supp}(\vec{x}^*)|, k_2 = |\textrm{Supp}(\vec{y}^*)|$ and $k = k_1+k_2$. Let $\vec{x}',\vec{y}'$ be the optimal solution to the min-max problem with payoff matrix the corresponding submatrix of payoff matrix $A$ (denoted by $B$) after removing the rows/columns which correspond to the coordinates that are not in the support of the unique optimal solution $(\vec{x}^*,\vec{y}^*)$\footnote{Note that $(\vec{x}',\vec{y}')$ should be the \textit{unique} optimal solution to the min-max problem with payoff matrix $B$.}. We consider the submatrix $\tilde{J}$ of the Jacobian matrix that is created by removing rows and columns of the corresponding coordinates that are not in the support of optimum (for the variables $\vec{x}$ and $\vec{y}$, these are exactly $n+m-k$). It is clear from above, that the Jacobian of OMWU has eigenvalues with absolute value less than one iff $\tilde{J}$ has as well. After also removing the rows (and the corresponding columns) that have only zero entries (these are exactly $n+m-k$, result zero eigenvalues and correspond
to variables $\vec{z}$ and $\vec{w}$) the resulting submatrix (denote it by $J$) boils down to the following:
\begin{equation}\label{eq:JOMWU}
J = \left(\begin{array}{cccc}
\vec{I}_{k_1 \times k_1} - D_x\vec{1}_{k_1}\vec{1}^{\top}_{k_1} & 2\eta D_x (B - v\vec{1}_{k_1}\vec{1}^{\top}_{k_2})  & \vec{0}_{k_1 \times k_1} & -\eta D_x (B - v\vec{1}_{k_1}\vec{1}^{\top}_{k_2})
\\ 2\eta D_y (v\vec{1}_{k_2}\vec{1}^{\top}_{k_1}-B^{\top}) & \vec{I}_{k_2 \times k_2} - D_y\vec{1}_{k_2}\vec{1}^{\top}_{k_2} & -\eta D_y (v\vec{1}_{k_2}\vec{1}^{\top}_{k_1}-B^{\top}) & \vec{0}_{k_2 \times k_2}
\\ \vec{I}_{k_1 \times k_1} & \vec{0}_{k_1 \times k_2} & \vec{0}_{k_1 \times k_1} & \vec{0}_{k_1 \times k_2}
\\ \vec{0}_{k_2 \times k_1} & \vec{I}_{k_2 \times k_2} & \vec{0}_{k_2 \times k_1} & \vec{0}_{k_2 \times k_2}
\end{array}\right).
\end{equation}
It is clear that $(\vec{1}_{k_1},\vec{0}_{k_2},\vec{0}_{k_1},\vec{0}_{k_2})$, $(\vec{0}_{k_1},\vec{1}_{k_2},\vec{0}_{k_1},\vec{0}_{k_2})$ are left eigenvectors with eigenvalues zero and thus any right eigenvector $(\tilde{\vec{x}},\tilde{\vec{y}},\tilde{\vec{z}},\tilde{\vec{w}})$ with nonzero eigenvalue has the property that $\tilde{\vec{x}}^{\top}\vec{1}_{k_1} = 0$ and $\tilde{\vec{y}}^{\top}\vec{1}_{k_2} = 0$.
Hence every nonzero eigenvalue of the matrix above is an eigenvalue of the matrix below:
\begin{equation}\label{eq:Jnew}
J_{\textrm{new}} = \left(\begin{array}{cccc}
\vec{I}_{k_1 \times k_1} & 2\eta D_x B  & \vec{0}_{k_1 \times k_1} & -\eta D_x B
\\ -2\eta D_y B^{\top} & \vec{I}_{k_2 \times k_2} & \eta D_y B^{\top} & \vec{0}_{k_2 \times k_2}
\\ \vec{I}_{k_1 \times k_1} & \vec{0}_{k_1 \times k_2} & \vec{0}_{k_1 \times k_1} & \vec{0}_{k_1 \times k_2}
\\ \vec{0}_{k_2 \times k_1} & \vec{I}_{k_2 \times k_2} & \vec{0}_{k_2 \times k_1} & \vec{0}_{k_2 \times k_2}
\end{array}\right).
\end{equation}
Let $p(\lambda)$ be the characteristic polynomial of the matrix (\ref{eq:Jnew}). After row/column operations it boils down to
\begin{equation}\label{eq:small}
(-1)^k\textrm{det}\left(\begin{array}{cc}
\lambda(1-\lambda) \vec{I}_{k_1 \times k_1} & (2\lambda - 1)\eta D_x B
\\ -\eta (2\lambda - 1)D_y B^{\top} & \lambda(1-\lambda) \vec{I}_{k_2 \times k_2}
\end{array}\right) = (1-2\lambda)^k q\left(\frac{\lambda(\lambda-1)}{2\lambda-1}\right),
\end{equation}
where $q(\lambda)$ is the characteristic polynomial of
\begin{equation}\label{eq:small}
J_{\textrm{small}}=\left(\begin{array}{cc}
\vec{0}_{k_1 \times k_1} & \eta D_x B
\\ -\eta D_y B^{\top} & \vec{0}_{k_2 \times k_2}
\end{array}\right).
\end{equation}
Observe that \[
J_{\textrm{small}} \cdot
\left(\begin{array}{cc}D_x & \vec{0}_{k_1 \times k_2}
\\ \vec{0}_{k_2 \times k_1} & D_y
\end{array}\right) \textrm{ is real skew symmetric,}\]
and hence by Lemma \ref{lem:skewsymmetrizable}, $J_{\textrm{small}}$ has eigenvalues of the form\footnote{We denote $i = \sqrt{-1}$.} $\pm i \eta\tau$ with $\tau \in \mathbb{R}$ (i.e., imaginary eigenvalues; we include $\eta$ in the expression to conclude that $\sigma := \eta \tau$ can be sufficiently small in absolute value). We conclude that any nonzero eigenvalue $\lambda$ of the matrix $J$ should satisfy the equation $\frac{\lambda(\lambda-1)}{2\lambda-1} = i \sigma$ for some small in absolute value $\sigma \in \mathbb{R}$. Finally we get that
\[\lambda = \frac{1+2i\sigma\pm \sqrt{1-4\sigma^2}}{2}.\]
We compute the square of the magnitude of $\lambda$ and we get $|\lambda|^2 = \frac{2-4\sigma^2 \pm 2\sqrt{1-4\sigma^2}+4\sigma^2}{4} = \frac{1\pm\sqrt{1-4\sigma^2}}{2} < 1$ unless $\sigma =0$ (i.e., $\tau=0$). If $\sigma=0$, it means that $J_{\textrm{new}}$ has an eigenvalue which is equal to one. Assume that $(\tilde{\vec{x}},\tilde{\vec{y}},\tilde{\vec{x}},\tilde{\vec{y}})$ is the corresponding right eigenvector, it holds that $B\tilde{\vec{y}} = \vec{0}$ and $B^{\top}\tilde{\vec{x}} = \vec{0}$. Assume also that there exists an eigenvalue that is equal to one in the original matrix $J$. It follows that $\vec{1}_{k_2}^{\top}\tilde{\vec{y}}=0$ and $\vec{1}_{k_1}^{\top}\tilde{\vec{x}}=0$. It holds that $\tilde{\vec{x}} = \vec{0}_{k_1}$ and $\tilde{\vec{y}} = \vec{0}_{k_2}$ otherwise $(\vec{x}',\vec{y}')+ t (\tilde{\vec{x}},\tilde{\vec{y}})$ would be another optimal solution (for the min-max problem with payoff matrix $B$; by padding zeros to the vector, we could create another optimal solution for the original min-max problem with payoff matrix $A$) for small enough $t$. We reached contradiction because we have assumed uniqueness.  Hence all the eigenvalues of $J$ are less than 1, i.e., the mapping is (locally) asymptotic stable mapping and the proof is complete.

\section{Experiments}\label{sec:experiments}
The purpose of our experiments is primarily to understand how the speed of convergence of OMWU dynamics (\ref{eq:expOMWU}) scales with the size of matrix $A$. Moreover, for $A$ of fixed size, we are interested in how the speed of convergence scales with the error of the output of OMWU dynamics. By error we mean the $\ell_1$ distance between the last iterate of OMWU and the optimal solution.

For the former case, we fix the error to be $0.1$ and we run OMWU for $n = 25 , 50 , ..., 250$ where the input matrix $A$ has size $n \times n$ with entries i.i.d random variables sampled from uniform $[-1,1]$. We output the number of iterations OMWU needs starting from uniform $(\frac{1}{n},...,\frac{1}{n})$ to reach a solution that is at most $0.1$ away from optimal in $\ell_1$ distance. We note that we computed the optimal solutions using LP-solvers.

For the latter case, we fix $n=50$ and we consider the error $\epsilon$ to be $\{0.5, 0.25, 0.0625, 0.015625 , 0.007812\}$. Starting from uniform distribution, we count the number of iterations to reach error $\epsilon$. The stepsize $\eta$ is fixed at $0.01$ at all times. The results can be found in the figure below (Figure \ref{fig:big}). If we had to guess, it seems that the relation between dimension and iterations is between linear and quadratic (i.e., OMWU dynamics has roughly cubic-quartic running time in $n$ if we count the cost of each iteration as quadratic) and the dependence between error $\epsilon$ and iterations $t$ seems like $t$ is inverse polynomial in $\epsilon$.

We note the importance of stepsize $\eta$. $\eta$ must be sufficiently small for our proofs to work. If $\eta$ is chosen to be big, then OMWU might not converge (might cycle, we observed such behavior in experiments). On the other hand, the smaller $\eta$ is chosen, the smaller the progress of OMWU dynamics (see the inequality claim for KL divergence) and hence the slower the dynamics.
\begin{figure}[h!]

\begin{minipage}{1.0\textwidth}
\centering     %%% not \center
\subfigure[In the $x$ axis we have the number of rows of a square matrix $A$ and on $y$ axis the number of iterations of OMWU. This figure captures how the number of iterations depends on the dimensionality of the min-max problem.]{\label{fig:GDAstable2} \includegraphics[width=0.47\linewidth]{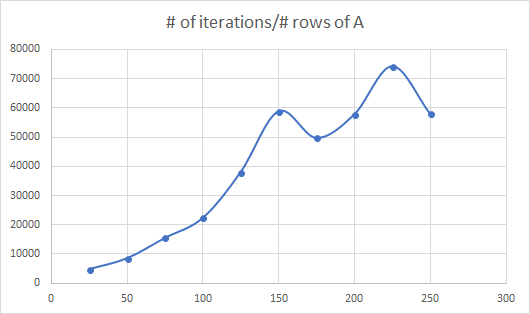}}
\;
\subfigure[In the $x$ axis we have the number of iterations of OMWU and on $y$ axis the $\ell_1$ distance from the optimal solution. This figure captures how the number of iterations scales with the error.]{\label{fig:GDAstable3}
  \includegraphics[scale=0.52]{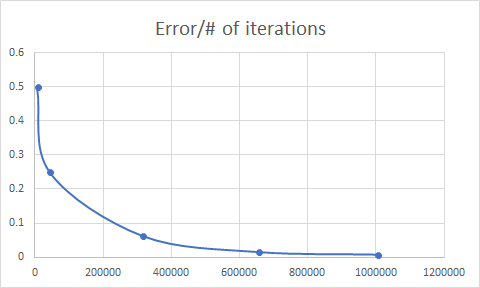}}
\end{minipage}\label{fig:big}
\end{figure}

%ACM is moving forward with the 2012 Classification system: http://dl.acm.org/ccs_flat.cfm. Please generate the CCSXML tex code through the online interactive system and insert the code below.

%\terms{Design, Algorithms, Performance}
%\keywords{NEED TO COMPLETE}

%\setcounter{page}{1}

\section{Conclusion}
In this paper we showed that a no-regret algorithm called Optimistic Multiplicative Weights Update (OMWU) converges pointwise to a Nash equilibrium in two player zero sum games (See also a concurrent work to ours \cite{Metal}, in which the authors provide a pointwise result about other dynamics, using different techniques). Our analysis is novel and does not follow the standard approaches of the literature of no-regret learning. We believe that our techniques can be useful in the analysis of other learning algorithms with no provable guarantees of pointwise convergence.

One interesting open question is to show that OMWU algorithm converges in polynomial time in $n,m$ (for proper choice of stepsize $\eta$) and find exact rates of convergence. Another possible future direction is to generalize our results about OMWU beyond the bilinear setting.
\section*{Acknowledgements}
We are grateful to Vasilis Syrgkanis for pointing out a mistake in Lemma B.4 (of the previous version, Lemma \ref{eq:nolabel} in this version) and suggesting how to fix it.
%Ioannis Panageas would like to acknowledge a MIT-SUTD postdoctoral fellowship.

%\newpage
\bibliographystyle{plain}
\bibliography{sigproc3}
\newpage
% Appendix to come
%\noindent
\appendix

\section{Equations of the Jacobian of OMWU dynamics}\label{sec:equationsJacobian}
\subsection{Equations computed at point $(\vec{x},\vec{y},\vec{z},\vec{w})$}
Set $S_x = \sum_{t=1}^n x_t e^{2\eta(A\vec{y})_t - \eta(A\vec{w})_t}$, $S_y = \sum_{t=1}^m y_t e^{-2\eta(A^{\top}\vec{x})_t + \eta(A^{\top}\vec{z})_t}$ and let $i,j$ be arbitrary indexes ($g_{1,i}$ captures the $i$-th coordinate of function $g_1$ etc),
\begin{equation}\label{eq:Jacobian}
\begin{array}{llll}
\frac{\partial g_{1,i}}{\partial x_i} = \frac{e^{2\eta(A\vec{y})_i - \eta(A\vec{w})_i}}{S_x}-x_i\frac{\left (e^{2\eta(A\vec{y})_i - \eta(A\vec{w})_i}\right)^2}{S_x^2} \textrm{ for all }i\in [n],\\
\frac{\partial g_{1,i}}{\partial x_j} = -x_i e^{2\eta(A\vec{y})_j - \eta(A\vec{w})_j} \cdot \frac{e^{2\eta(A\vec{y})_i - \eta(A\vec{w})_i}}{S_x^2} \textrm { for all }i\in [n], j \in [n] \textrm{ and }j\neq i,\\
\frac{\partial g_{1,i}}{\partial y_j} = x_ie^{2\eta(A\vec{y})_i - \eta(A\vec{w})_i} \cdot \frac{2\eta A_{ij} S_x - 2\eta \sum_{t} A_{tj}x_t e^{2\eta(A\vec{y})_t - \eta(A\vec{w})_t}}{S_x^2} \textrm { for all }i\in [n], j \in [m],\\
\frac{\partial g_{1,i}}{\partial z_j} = 0 \textrm { for all }i,j \in [n],\\
\frac{\partial g_{1,i}}{\partial w_j} = x_ie^{2\eta(A\vec{y})_i - \eta(A\vec{w})_i} \cdot \frac{-\eta A_{ij} S_x + \eta \sum_{t} A_{tj}x_t e^{2\eta(A\vec{y})_t - \eta(A\vec{w})_t}}{S_x^2} \textrm { for all }i \in [n],j \in [m],\\
\frac{\partial g_{2,i}}{\partial y_i} = \frac{e^{-2\eta(A^{\top}\vec{x})_i + \eta(A^{\top}\vec{z})_i}}{S_y}-y_i\frac{\left (e^{-2\eta(A^{\top}\vec{x})_i + \eta(A^{\top}\vec{z})_i}\right)^2}{S_y^2} \textrm{ for all }i \in [m],\\
\frac{\partial g_{2,i}}{\partial y_j} = -y_i e^{-2\eta(A^{\top}\vec{x})_j + \eta(A^{\top}\vec{z})_j} \cdot \frac{e^{-2\eta(A^{\top}\vec{x})_i + \eta(A^{\top}\vec{z})_i}}{S_y^2} \textrm { for all }i\in [m], j \in [m] \textrm{ and }j\neq i,\\
\frac{\partial g_{2,i}}{\partial x_j} = y_ie^{-2\eta(A^{\top}\vec{x})_i + \eta(A^{\top}\vec{z})_i} \cdot \frac{-2\eta A^{\top}_{ij} S_y + 2\eta \sum_{t} A^{\top}_{tj}y_t e^{-2\eta(A^{\top}\vec{x})_t + \eta(A^{\top}\vec{z})_t}}{S_y^2} \textrm { for all }i\in [m], j \in [n],\\
\frac{\partial g_{2,i}}{\partial z_j} = y_ie^{-2\eta(A^{\top}\vec{x})_i + \eta(A^{\top}\vec{z})_i} \cdot \frac{\eta A^{\top}_{ij} S_y - \eta \sum_{t} A^{\top}_{tj}x_t e^{-2\eta(A^{\top}\vec{x})_t + \eta(A^{\top}\vec{z})_t}}{S_y^2} \textrm { for all }i\in [m], j \in [n],\\
\frac{\partial g_{2,i}}{\partial w_j} = 0 \textrm { for any }i,j \in [m],\\
\frac{\partial g_{3,i}}{\partial x_i} = 1 \textrm{ for all }i\in [n] \textrm{ and zero all the other partial derivatives of }g_{3,i},\\
\frac{\partial g_{4,i}}{\partial y_i} = 1 \textrm{ for all }i\in [m] \textrm{ and zero all the other partial derivatives of }g_{4,i}.
\end{array}
\end{equation}
\newpage
\subsection{Equations computed at point $(\vec{x}^*,\vec{y}^*,\vec{x}^*,\vec{y}^*)$}
Set $S_x = \sum_{t=1}^n x_t^{*} e^{\eta(A\vec{y}^*)_t}$, $S_y = \sum_{t=1}^m y_t^{*} e^{-\eta(A^{\top}\vec{x}^*)_t}$ and let $i,j$ be arbitrary indexes ($g_{1,i}$ captures the $i$-th coordinate of function $g_1$ etc). Assume $v = \vec{x}^{* \top} A \vec{y}^*$, it is not hard to see that $(A^{\top}\vec{x}^*)_i = (A\vec{y}^*)_j = v$ for all $i \in \textrm{Supp}(\vec{x}^*), j \in \textrm{Supp}(\vec{y}^*)$ and $S_x = e^{\eta v}, S_y = e^{-\eta v}$. We get that:
\begin{equation}\label{eq:Jacobian2}
\begin{array}{llll}
\frac{\partial g_{1,i}}{\partial x_i} = 1-x^*_i \textrm{ for all }i\in \textrm{Supp}(\vec{x}^*),\\
\frac{\partial g_{1,i}}{\partial x_i} = \frac{e^{\eta(A\vec{y}^*)_i}}{e^{\eta v}} \textrm{ for all }i\notin  \textrm{Supp}(\vec{x}^*),\\
\frac{\partial g_{1,i}}{\partial x_j} = -x^*_i \textrm { for all }i, j \in \textrm{Supp}(\vec{x}^*) \textrm{ and }j\neq i,\\
\frac{\partial g_{1,i}}{\partial x_j} = 0 \textrm { for all }i\notin \textrm{Supp}(\vec{x}^*), j \in [n] \textrm{ and }j\neq i,\\
\frac{\partial g_{1,i}}{\partial y_j} = x_i^* (2\eta A_{ij} - 2\eta v)  \textrm { for all }i\in \textrm{Supp}(\vec{x}^*), j \in \textrm{Supp}(\vec{y}^*),\\
\frac{\partial g_{1,i}}{\partial y_j} = 0 \textrm { for all }i\notin \textrm{Supp}(\vec{x}^*), j \in [m],\\
\frac{\partial g_{1,i}}{\partial z_j} = 0 \textrm { for all }i,j \in [n],\\
\frac{\partial g_{1,i}}{\partial w_j} = x_i^* (-\eta A_{ij} + \eta v ) \textrm { for all }i\in \textrm{Supp}(\vec{x}^*),j \in \textrm{Supp}(\vec{y}^*),\\
\frac{\partial g_{1,i}}{\partial w_j} = 0 \textrm { for all }i\notin \textrm{Supp}(\vec{x}^*),j \in [m],\\
\frac{\partial g_{2,i}}{\partial y_i} = 1-y^*_i \textrm{ for all }i\in \textrm{Supp}(\vec{y}^*),\\
\frac{\partial g_{2,i}}{\partial y_i} = \frac{e^{-\eta(A\vec{x}^*)_i}}{e^{-\eta v}} \textrm{ for all }i\notin  \textrm{Supp}(\vec{y}^*),\\
\frac{\partial g_{2,i}}{\partial y_j} = -y^*_i \textrm { for all }i, j \in \textrm{Supp}(\vec{y}^*) \textrm{ and }j\neq i,\\
\frac{\partial g_{2,i}}{\partial y_j} = 0 \textrm { for all }i\notin \textrm{Supp}(\vec{y}^*), j \in [m] \textrm{ and }j\neq i,\\
\frac{\partial g_{2,i}}{\partial x_j} = y_i^* (-2\eta A_{ij}^{\top} + 2\eta v)  \textrm { for all }i\in \textrm{Supp}(\vec{y}^*), j \in \textrm{Supp}(\vec{x}^*),\\
\frac{\partial g_{2,i}}{\partial x_j} = 0 \textrm { for all }i\notin \textrm{Supp}(\vec{y}^*), j \in [n],\\
\frac{\partial g_{2,i}}{\partial z_j} = y_i^* (\eta A^{\top}_{ij} - \eta v) \textrm { for all }i \in \textrm{Supp}(\vec{y}^*), j \in \textrm{Supp}(\vec{x}^*),\\
\frac{\partial g_{2,i}}{\partial z_j} = 0 \textrm { for all }i\notin \textrm{Supp}(\vec{y}^*), j \in [n],\\
\frac{\partial g_{2,i}}{\partial w_j} = 0 \textrm { for any }i,j \in [m],\\
\frac{\partial g_{3,i}}{\partial x_i} = 1 \textrm{ for all }i\in [n] \textrm{ and zero all the other partial derivatives of }g_{3,i},\\
\frac{\partial g_{4,i}}{\partial y_i} = 1 \textrm{ for all }i\in [m] \textrm{ and zero all the other partial derivatives of }g_{4,i}.
\end{array}
\end{equation}

\section{Missing claims and proofs}
Lemma \ref{lem:sameorder} shows that the change between next and current iterate in both OMWU algorithms (classic and linear variant) is of order $O(\eta)$ and that the difference between the next iterate of both algorithms is $O(\eta^2)$.
\begin{lemma}\label{lem:sameorder} Let $\vec{x} \in \Delta_n$ be the vector of the max player, $\vec{w},\vec{z} \in \Delta_m$ and suppose $\vec{x}', \vec{x}''$ are the next iterates of OMWU and its linear variant with current vector $\vec{x}$ and vectors $\vec{w},\vec{z}$ of the min player. It holds that
\[\norm[1]{\vec{x}' - \vec{x}''} \textrm{ is }O(\eta^2), \textrm{and }\norm[1]{\vec{x}' - \vec{x}}, \norm[1]{\vec{x}'' - \vec{x}} \textrm{ are }O(\eta).\] Analogously, it holds for vector $\vec{y} \in \Delta_m$ of the min player and its next iterates.
\end{lemma}
\begin{proof} Let $\eta$ be sufficiently small (smaller than maximum in absolute value entry of $A$).
\begin{align*}
|x'_i - x''_i| &= x_i \left |\frac{e^{2\eta (A\vec{w})_i - \eta (A\vec{z})_i}}{\sum_j x_j e^{2\eta (A\vec{w})_j - \eta (A\vec{z})_j}}  -  \frac{1+2\eta (A\vec{w})_i - \eta (A\vec{z})_i}{\sum_j x_j (1+2\eta (A\vec{w})_j - \eta (A\vec{z})_j)}\right|
\\&= x_i \left | \frac{1+2\eta (A\vec{w})_i - \eta (A\vec{z})_i \pm O(\eta^2)}{\sum_j x_j (1+2\eta (A\vec{w})_j - \eta (A\vec{z})_j) \pm O(\eta^2)}  -  \frac{1+2\eta (A\vec{w})_j - \eta (A\vec{z})_j}{\sum_j x_j (1+2\eta (A\vec{w})_j - \eta (A\vec{z})_j)}\right|
\\&\textrm{ which is }O(\eta^2)x_i
\end{align*}
and hence $\norm[1]{\vec{x}'-\vec{x}''}$ is $O(\eta^2)$. Moreover we have that
\begin{align*}
|x_i - x''_i| &= x_i \left |1  -  \frac{1+2\eta (A\vec{w})_i - \eta (A\vec{z})_i}{\sum_j x_j (1+2\eta (A\vec{w})_i - \eta (A\vec{z})_i)}\right|
\\&= x_i \left | \frac{\sum_j x_j(1+2\eta (A\vec{w})_j - \eta (A\vec{z})_j) - (1+2\eta (A\vec{w})_i - \eta (A\vec{z})_i)}{\sum_j x_j (1+2\eta (A\vec{w})_i - \eta (A\vec{z})_i) }  \right|
\\& = x_i \left | \frac{\sum_j x_j(2\eta (A\vec{w})_j - \eta (A\vec{z})_j ) - 2\eta (A\vec{w})_i + \eta (A\vec{z})_i }{\sum_j x_j (1+2\eta (A\vec{w})_i - \eta (A\vec{z})_i) }  \right|
\textrm{ which is }O(\eta) x_i.
\end{align*}
By triangle inequality and the two above proofs we get the third part of the lemma.
\end{proof}
Lemmas \ref{lem:mainsub}, \ref{eq:nolabel} and \ref{lem:ineqfixed} will be used in the proof of Theorem \ref{thm:kl}.
\begin{lemma}\label{lem:mainsub} Let $(\vec{x}^t,\vec{y}^t)$ be the $t$-th iterate of OMWU dynamics (\ref{eq:expOMWU}). We set \\$R_{\vec{x}}^t := \sum_ix^t_i\left(2(A\vec{y}^t)_i-(A\vec{y}^{t-1})_i-2\vec{x}^{t\;\top}A\vec{y}^t+\vec{x}^{t\;\top}A\vec{y}^{t-1}\right)^2$ and \\ $R_{\vec{y}}^t := \sum_iy^t_i\left(2(A^{\top}\vec{x}^t)_i-(A^{\top}\vec{x}^{t-1})_i-2\vec{x}^{t\;\top}A\vec{y}^t+\vec{x}^{t-1\;\top}A\vec{y}^t\right)^2$. For each time step $t \geq 2$ it holds 
\begin{equation}
\begin{array}{cc}
\textrm{$1)$ }\vec{x}^{t-1\;\top}A(2\vec{y}^{t}-\vec{y}^{t-1}) - \vec{x}^{t\;\top}A(2\vec{y}^{t}-\vec{y}^{t-1})  \leq -(1-O(\eta))\eta R_{\vec{x}}^t+O(\eta^2),\\
\textrm{$2)$ }\vec{y}^{t\;\top}A^{\top}(2\vec{x}^{t}-\vec{x}^{t-1}) - \vec{y}^{t-1\;\top}A^{\top}(2\vec{x}^{t}-\vec{x}^{t-1})\leq -(1-O(\eta))\eta R_{\vec{y}}^t+O(\eta^2)
\end{array}
\end{equation}
\end{lemma}
\begin{proof}
We are going to prove the first inequality. Analogously we can prove the second inequality. It suffices to prove 
\[
\vec{x}^{t-1\;\top}A(2\vec{y}^{t}-\vec{y}^{t-1}) - \tilde{\vec{x}}^{t\;\top}A(2\vec{y}^{t}-\vec{y}^{t-1})  \leq -(1-O(\eta))\eta R_{\vec{x}}^t+O(\eta^2),
\]
where iterate $\tilde{\vec{x}}^t$ is the update of $\vec{x}^{t-1}$ using the linear variant of OMWU dynamics. This is true due to Lemma \ref{lem:sameorder}, i.e., because $\norm[1]{\vec{x}^t-\tilde{\vec{x}}^t}$ is $O(\eta^2)$ in distance. Moreover, due to Lemma \ref{lem:sameorder}, i.e., because $\norm[1]{\vec{x}^t-\vec{x}^{t-1}}$ is $O(\eta)$, it suffices to prove that 
\begin{align*}
\vec{x}^{t-1\;\top}&A(2\vec{y}^{t}-\vec{y}^{t-1}) - \tilde{\vec{x}}^{t\;\top}A(2\vec{y}^{t}-\vec{y}^{t-1})  \\&\leq -(1-O(\eta))\eta \sum_ix^{t-1}_i\left(2(A\vec{y}^t)_i-(A\vec{y}^{t-1})_i-2\vec{x}^{t-1\;\top}A\vec{y}^t+\vec{x}^{t-1\;\top}A\vec{y}^{t-1}\right)^2+O(\eta^2)
\end{align*}
Observe now by plugging in the update rule of $\tilde{\vec{x}}^t$ we get
\begin{align*}
&\vec{x}^{t-1\;\top}A(2\vec{y}^{t}-\vec{y}^{t-1}) - \tilde{\vec{x}}^{t\;\top}A(2\vec{y}^{t}-\vec{y}^{t-1}) \\& =  \eta\frac{\sum_{i}x_i^{t-1}(A(2\vec{y}^{t}-\vec{y}^{t-1}))_i (A(2\vec{y}^{t-1}-\vec{y}^{t-2}))_i-  \left[\vec{x}^{t-1\;\top}A(2\vec{y}^{t}-\vec{y}^{t-1})\right]\left[\vec{x}^{t-1 \;\top}A(2\vec{y}^{t-1} - \vec{y}^{t-2})\right]}{1+\eta \vec{x}^{t-1 \;\top}A(2\vec{y}^{t-1} - \vec{y}^{t-2})}\\&= \eta\frac{ \sum_{i}x_i^{t-1}(A(2\vec{y}^{t}-\vec{y}^{t-1}))^2_i -\left[\vec{x}^{t-1\;\top}A(2\vec{y}^{t}-\vec{y}^{t-1})\right]^2}{1+\eta \vec{x}^{t-1 \;\top}A(2\vec{y}^{t-1} - \vec{y}^{t-2})}+O(\eta^2),\\
\end{align*}
where the last equality uses Lemma \ref{lem:sameorder} and first equality is just calculations. Observe that the denominator is of order $1+O(\eta)$, and the numerator is equal to $$-\eta\sum_ix^{t-1}_i\left(2(A\vec{y}^t)_i-(A\vec{y}^{t-1})_i-2\vec{x}^{t-1\;\top}A\vec{y}^t+\vec{x}^{t-1\;\top}A\vec{y}^{t-1}\right)^2$$ due to the variance formula used on the random variable $z$ where $z = \eta (A(2\vec{y}^t-\vec{y}^{t-1}))_i$ with probability $x_i^t.$ The claim follows. 
\end{proof}
\begin{lemma}\label{eq:nolabel} Let $(\vec{x}^t,\vec{y}^t)$ be the $t$-th iterate of OMWU dynamics (\ref{eq:expOMWU}). We set \\$R_{\vec{x}}^t := \sum_ix^t_i\left(2(A\vec{y}^t)_i-(A\vec{y}^{t-1})_i-2\vec{x}^{t\;\top}A\vec{y}^t+\vec{x}^{t\;\top}A\vec{y}^{t-1}\right)^2$ and \\ $R_{\vec{y}}^t := \sum_iy^t_i\left(2(A^{\top}\vec{x}^t)_i-(A^{\top}\vec{x}^{t-1})_i-2\vec{x}^{t\;\top}A\vec{y}^t+\vec{x}^{t-1\;\top}A\vec{y}^t\right)^2$. For each time step $t \geq 2$ it holds
\begin{align*}
&\eta \vec{x}^{t-1 \ \top} A\vec{y}^t  -  \eta \vec{x}^{t \ \top}A\vec{y}^{t-1} 
\leq - (1-O(\eta))\eta^2R_{\vec{x}}^t -(1-O(\eta))\eta^2R_{\vec{y}}^t + O(\eta^3).
\end{align*}
\end{lemma}
\begin{proof}
Summing the two inequalities from Lemma \ref{lem:mainsub}, we get that 
$$\vec{x}^{t-1 \;\top}A\vec{y}^t - \vec{y}^{t-1 \;\top}A^{\top}\vec{x}^t \leq -(1-O(\eta))\eta (R_{\vec{x}}^t+R_{\vec{y}}^t)+O(\eta^2).$$
Multiplying with $\eta$ both sides, the claim follows.
\end{proof}
\begin{lemma} \label{lem:ineqfixed} Let $(\vec{x}^t, \vec{y}^t)$ denote the $t$-th iterate of OMWU dynamics. It holds for $t \geq 2$ that \[\vec{x}^{* \ \top}A(2\vec{y}^t - \vec{y}^{t-1}) \geq \vec{x}^{* \ \top}A\vec{y}^* \textrm{ and } (2\vec{x}^{t \ \top} - \vec{x}^{t-1 \ \top})A\vec{y}^* \leq \vec{x}^{* \ \top}A\vec{y}^*,\]
where $(\vec{x}^*,\vec{y}^*)$ is the optimal solution of the min-max problem.
\end{lemma}
\begin{proof}
It is true that $x_i^t \geq (1-O(\eta))x_i^{t-1}$, hence $x_i^t \geq \frac{1}{2}x_i^{t-1}$ for $\eta$ sufficiently small. Therefore $2\vec{x}^t - \vec{x}^{t-1}$ lies in the simplex $\Delta_n$. Hence since $(\vec{x}^*,\vec{y}^*)$ is the optimum (Nash equilibrium) we get that
$(2\vec{x}^{t \ \top} - \vec{x}^{t-1 \ \top})A\vec{y}^* \leq \vec{x}^{* \ \top}A\vec{y}^*$ ($\vec{x}$ is the max player). Similarly the second inequality can be proved.
\end{proof}

\begin{proof}[Proof of Theorem \ref{thm:kl}]
We compute the difference between\\ $D_{KL}((\vec{x}^*,\vec{y}^*) || (\vec{x}^{t+1},\vec{y}^{t+1}))$ and $D_{KL}((\vec{x}^*,\vec{y}^*) || (\vec{x}^{t},\vec{y}^{t}))$
\begin{align*}
&D_{KL}((\vec{x}^*,\vec{y}^*) || (\vec{x}^{t+1},\vec{y}^{t+1})) -  D_{KL}((\vec{x}^*,\vec{y}^*) || (\vec{x}^{t},\vec{y}^{t})) = -\left (\sum_{i}\vec{x}_i^* \ln \frac{x_i^{t+1}}{x_i^{t}} + \sum_{i}y_i^*\ln \frac{y_i^{t+1}}{y_i^{t}} \right )
\\ =& -\left (\sum_{i}\vec{x}_i^* \ln e^{2\eta (A\vec{y}^t)_i-\eta (A\vec{y}^{t-1})_i} + \sum_{i}y_i^*\ln e^{-2\eta(A^{\top}\vec{x}^{t})_i+\eta(A^{\top}\vec{x}^{t-1})_i} \right )
\\&+ \ln \left(\sum_i x_i^{t}e^{2\eta (A\vec{y}^t)_i-\eta (A\vec{y}^{t-1})_i}\right) + \ln \left(\sum_i y_i^{t}e^{-2\eta (A^{\top}\vec{x}^t)_i+\eta (A^{\top}\vec{x}^{t-1})_i}\right)
\\=& {-2\eta \vec{x}^{* \ \top} A \vec{y}^t + \eta \vec{x}^{* \ \top} A \vec{y}^{t-1} + 2\eta \vec{x}^{t \ \top} A \vec{y}^* - \eta \vec{x}^{t-1 \ \top} A \vec{y}^*} 
\\&+ \ln \left(\sum_i x_i^{t}e^{2\eta (A\vec{y}^t)_i-\eta (A\vec{y}^{t-1})_i}\right) + \ln \left(\sum_i y_i^{t}e^{-2\eta (A^{\top}\vec{x}^t)_i+\eta (A^{\top}\vec{x}^{t-1})_i}\right).
\end{align*}
We use Lemma \ref{lem:ineqfixed} and we get that ${-2\eta \vec{x}^{* \ \top} A \vec{y}^t + \eta \vec{x}^{* \ \top} A \vec{y}^{t-1} + 2\eta \vec{x}^{t \ \top} A \vec{y}^* - \eta \vec{x}^{t-1 \ \top} A \vec{y}^*} \leq 0$, therefore the LHS (difference in the KL divergence) is at most

\begin{align*}
\leq & \overbrace{-2\eta \vec{x}^{t \ \top} A \vec{y}^t + \eta \vec{x}^{t \ \top} A \vec{y}^{t-1} + 2\eta \vec{x}^{t \ \top} A \vec{y}^t - \eta \vec{x}^{t-1 \ \top} A \vec{y}^{t} - \eta \vec{x}^{t \ \top} A \vec{y}^{t-1} + \eta \vec{x}^{t-1 \ \top} A \vec{y}^{t}}^{=0}
\\&+ \ln \left(\sum_i x_i^{t}e^{2\eta (A\vec{y}^t)_i-\eta (A\vec{y}^{t-1})_i}\right) + \ln \left(\sum_i y_i^{t}e^{-2\eta (A^{\top}\vec{x}^t)_i+\eta (A^{\top}\vec{x}^{t-1})_i}\right)
\\=& \ln \left(\sum_i x_i^{t}e^{2\eta ((A\vec{y}^t)_i - \vec{x}^{t \ \top}A\vec{y}^t)-\eta ((A\vec{y}^{t-1})_i- \vec{x}^{t \ \top}A\vec{y}^{t-1})}\right) \\&+ \ln \left(\sum_i y_i^{t}e^{-2\eta ((A^{\top}\vec{x}^t)_i- \vec{x}^{t \ \top}A\vec{y}^t)+\eta ((A^{\top}\vec{x}^{t-1})_i- \vec{x}^{t-1 \ \top}A\vec{y}^{t})}\right)- \eta \vec{x}^{t \ \top} A \vec{y}^{t-1} + \eta \vec{x}^{t-1 \ \top} A \vec{y}^{t}
\end{align*}
We furthermore use second order Taylor approximation ($\eta$ is sufficiently small) to the function $e^x$ and we get that previous expression is at most
\begin{align*}
&\leq \ln (\sum_i x_i^{t}\left(1+2\eta ((A\vec{y}^t)_i -\vec{x}^{t \ \top}A\vec{y}^t)-\eta ((A\vec{y}^{t-1})_i -\vec{x}^{t \ \top}A\vec{y}^{t-1})\right)
\\&+ \sum_i x_i^{t} \left((\frac{1}{2} + O(\eta))\eta^2 \left(2 (A\vec{y}^t)_i - 2 \vec{x}^{t \ \top}A\vec{y}^t -(A\vec{y}^{t-1})_i+ \vec{x}^{t \ \top}A\vec{y}^{t-1}\right)^2\right))
\\&+ \ln (\sum_i y_i^{t}\left(1-2\eta ((A^{\top}\vec{x}^t)_i -\vec{x}^{t \ \top}A\vec{y}^t)+\eta ((A^{\top}\vec{x}^{t-1})_i -\vec{x}^{t-1 \ \top}A\vec{y}^{t})\right)
\\&+ \sum_i y_i^{t} \left((\frac{1}{2} + O(\eta))\eta^2 \left(2 (A^{\top}\vec{x}^t)_i - 2 \vec{x}^{t \ \top}A\vec{y}^t -(A^{\top}\vec{x}^{t-1})_i+ \vec{x}^{t-1 \ \top}A\vec{y}^{t}\right)^2\right)) \\& -\eta \vec{x}^{t \ \top} A \vec{y}^{t-1} + \eta \vec{x}^{t-1 \ \top} A \vec{y}^{t}
\end{align*}

Finally, using Taylor approximation on $\log (1+x)$ and Lemma \ref{eq:nolabel} (last inequality) we get the following system:
\begin{align*}
\\&\leq \eta \vec{x}^{t-1 \ \top} A\vec{y}^t -  \eta \vec{x}^{t \ \top}A\vec{y}^{t-1} 
\\&+ \sum_i x_i^{t} \left((\frac{1}{2} + O(\eta))\eta^2 \left(2 (A\vec{y}^t)_i - 2 \vec{x}^{t \ \top}A\vec{y}^t -(A\vec{y}^{t-1})_i+ \vec{x}^{t \ \top}A\vec{y}^{t-1}\right)^2\right)
\\&+ \sum_i y_i^{t} \left((\frac{1}{2} + O(\eta))\eta^2 \left(2 (A^{\top}\vec{x}^t)_i - 2 \vec{x}^{t \ \top}A\vec{y}^t -(A^{\top}\vec{x}^{t-1})_i+ \vec{x}^{t-1 \ \top}A\vec{y}^{t}\right)^2\right)
\\& \leq -\sum_i x_i^{t} \left((\frac{1}{2} - O(\eta))\eta^2 \left(2 (A\vec{y}^t)_i - 2 \vec{x}^{t \ \top}A\vec{y}^t -(A\vec{y}^{t-1})_i+ \vec{x}^{t \ \top}A\vec{y}^{t-1}\right)^2\right)
\\& -\sum_i y_i^{t} \left((\frac{1}{2} - O(\eta))\eta^2 \left(2 (A^{\top}\vec{x}^t)_i - 2 \vec{x}^{t \ \top}A\vec{y}^t -(A^{\top}\vec{x}^{t-1})_i+ \vec{x}^{t-1 \ \top}A\vec{y}^{t}\right)^2\right)+O(\eta^3).
\end{align*}
It is clear that as long as $(\vec{x}^t,\vec{y}^t)$ (and thus $(\vec{x}^{t-1},\vec{y}^{t-1})$ by Lemma~\ref{lem:sameorder}) is not $O({\eta}^{1/3})$-close, from above inequalities/equalities we get \[D_{KL}((\vec{x}^*,\vec{y}^*) || (\vec{x}^{t+1},\vec{y}^{t+1})) - D_{KL}((\vec{x}^*,\vec{y}^*) || (\vec{x}^{t},\vec{y}^{t})) \leq -\Omega(\eta^3),\] meaning that KL divergence decreases by at least a factor of $\eta^3$ and the claim follows.
\end{proof}

\begin{lemma}\label{lem:skewsymmetrizable} Let $D$ be a real diagonal matrix with positive diagonal entries and $S$ be a real skew-symmetric matrix ($S^{\top} = -S$). It holds that $SD$ has eigenvalues with real part zero (i.e., it has only imaginary eigenvalues).
\end{lemma}
\begin{proof} Let $\vec{z}^*$ be the conjugate transpose of $\vec{z}$ and $\vec{z}^*$ be a left eigenvector of $SD$ with complex eigenvalue $\lambda$. It holds that
\begin{align*}
\lambda \vec{z}^{*}D^{-1}\vec{z} &=  \vec{z}^{*}SD D^{-1}\vec{z}
\\& = \vec{z}^{*}S\vec{z}
\\& = - (\vec{z}^{*}S\vec{z})^{*} \textrm{ (since }S \textrm{ is skew symmetric)}
\\& = - (\lambda \vec{z}^{*}D^{-1}\vec{z})^{*} \textrm{ (using first and second equalities above)}
\\& = -\overline{\lambda}\vec{z}^{*}D^{-1} \vec{z}.
\end{align*}
Since $D$ has positive diagonal entries, we conclude that $\vec{z}^{*}D^{-1}\vec{z} \neq 0$ (since $\vec{z} \neq \vec{0}$), thus $\lambda = -\overline{\lambda}$ and the claim follows.
\end{proof}

\end{document}